\documentclass[12pt, reqno]{amsart}
 \usepackage{tikz}
 \usepackage{hyperref, mathtools,mathrsfs}
 \usepackage[skip=4pt]{caption}
 \usepackage{amssymb,amsmath,graphicx,amsthm}
  \usetikzlibrary{positioning}
  \usepackage{tkz-euclide}
\usepackage{rotating}
\usepackage{color,xcolor}
\usepackage{geometry}
\definecolor{darkred}{rgb}{1,0,0} 
\definecolor{darkgreen}{rgb}{0,0.8,0}
\definecolor{darkblue}{rgb}{0,0,1}

\hypersetup{colorlinks,
linkcolor=darkblue,
filecolor=darkgreen,
urlcolor=darkred,
citecolor=darkgreen}

 \def\f{\mathtt{f}}

 \def\bt{\begin{theorem}}
 	\def\el{\end{lemma}}
 \def\bl{\begin{lemma}}
 	\def\et{\end{theorem}}
 \def\bp{\begin{proposition}}
 	\def\ep{\end{proposition}}
 \def\bd{\begin{definition}}
 	\def\ed{\end{definition}}
 \def\br{\begin{remark}}
 	\def\er{\end{remark}}

 \def\R{{\mathbb R}}

 \def\C{\mathbb C}
 \def\B{\mathbb B}
 
 \def\N{{\mathbb N}}

\def\S{{\mathbb S}}
\def\S1{{\mathbb S^1}}

\def\D{{\mathbb D}}

\def\f{{\bold f}}

 \def\dib{\bar\partial}
 \def\label#1{\label{#1}{\bf(#1)}~}

 \numberwithin{equation}{section}

 \theoremstyle{plain}
 \newtheorem*{theorem*}{Theorem}
 \newtheorem{example}{Example}[section]
 
 \newtheorem{theorem}{Theorem}[section]
 
 \newtheorem{corollary}[theorem]{Corollary}
 
 \newtheorem{lemma}[theorem]{Lemma}
 
 \newtheorem{proposition}[theorem]{Proposition}

 \theoremstyle{definition}
 
 \newtheorem{definition}[theorem]{Definition}
 
 \theoremstyle{remark}

 \newtheorem{remark}[theorem]{Remark}

 \newcommand{\p}{\partial}

 \theoremstyle{plain} 

 \def\Re{\text{Re}}
  \def\Im{\text{Im}}

\begin{document}
 	
 	\title{Testing families of analytic discs  in the unit ball}  
\author{Luca Baracco}
\address{Dipartimento di Matematica Tullio Levi-Civita, Universit\`a di Padova, via Trieste 63, 35121 Padova, Italy}
\email{baracco@math.unipd.it}

\author{Stefano Pinton}
\address{Dipartimento di Matematica, Politecnico di Milano, via E. Bonardi 9, 20133 Milano,, Italy}
\email{stefano.pinton@polimi.it}

 \begin{abstract} Let $a,b,c\in \C^2$ be three non collinear points such that their mutual joining complex lines do not intersect the unit ball $\B^2$ and such that the line through $a$ and $b$ is tangent to $\B^2$. Then the set of lines concurrent to $a,b$ and $c$ is a testing family for continuous functions on $\mathbb{S}^3$. This improves a result by the authors and solves a case left open in the literature as described by Globevnik.  
	
\end{abstract}
  	\maketitle
  	\maketitle

 \tableofcontents 	
 	\section{Introduction and main results}
One of the most striking properties of holomorphic functions in several variables is the fact that the property of being holomorphic can be checked separately on each variable while the others are kept as constants. This fact is known as Hartogs Theorem. Another remarkable property, which descend for instance from integral representation formulas, is that these functions are uniquely determined by their value on smaller sets such as the boundary of a domain or more generally on a generic manifold (see section \ref{S3} for the definition).
It is natural to ask if these two properties can be combined or in other words:Given a function $f$ on the boundary of a domain $\Omega\subset \C^n$ such that $f$ is holomorphic when restricted on a set of complex curves (like the coordinate lines) is then $f$ the boundary value of a holomorphic function $F\in \text{Hol}(\Omega)$ ? This last question needs to be clarified because the function is defined on a set of dimension smaller than $2n$ and so the restriction on complex curves or with respect to a group of variables might not have sense. We shall denote by $\D=\{ \tau\in \C |\ |\tau|<1\}$ the standard unit disc in $\C$, by $\overline{\D}$ its closure, $\p \D$ its boundary. For a general disc of radius $R$ and center $c$ we write $\D_{c,R}=\{\tau\in\C|\ |\tau-c|<R\}$.
\begin{definition}
	An analytic disc in $\C^n$ for $n>1$ is a continuous, injective  map $A:\overline{\D}\rightarrow \C^n$ which is holomorphic on $\D$.
	The set $A(\p\D)$ is called the boundary of $A$ and is denoted by $\p A$. We say that $A$ is attached to a set $M$ if $\p A \subset M$.  
\end{definition}          
\begin{definition}
	Let $M\subset \C^n$ and $f:M\rightarrow \C$ a continuous function. If $A$ is an analytic disc attached to $M$ we say that $f$ extends holomorphically on $A$ if there exists a continuous function $F:\overline{\D} \rightarrow \C$  holomorphic on $\D$ such that $f(A(\tau))=F(\tau)$ for $\tau\in \p\D$.  
\end{definition}
We will use the term "disc" for both the map $A$ and also the image $A(\D)$. We will say that a disc is straight or a line if $A(\D)$ is contained in a complex line. Given a smooth domain $\Omega\subset \C^n$ and a function $f:\p\Omega\rightarrow \C$ by saying that $f$ extends holomorphically on a complex line $l(\tau)=a\tau +b$, $a,b\in\C^n$ we mean that
if $D:=\{\tau\in \C |\  l(\tau) \in \Omega\}$ is a bounded simply connected domain and if $\phi :\D \rightarrow D$ is the Riemann's map, then the function $f\circ \phi :\p\D \rightarrow \C$ extends holomorphically. If $F$ is the holomorphic extension of $f\circ \phi$ we will refer to the extension of $f$ on the line $l$ as the function on the analytic disc $A=l\circ \phi$ (i.e. on the image of the map $A$) $f(A(\tau)) :=F(\tau)$.
Let $\Lambda\subset \R^m$ be a set of parameters, a family of discs on $\Lambda$ is a  continuous map $A:\bar{\D}\times \Lambda\rightarrow\C^n$, $A(\tau,\lambda)=A_{\lambda}(\tau)$ which is holomorphic in $\tau$. 
\begin{definition}
	Let $A_\lambda$, $\lambda\in\Lambda$ be a family of analytic discs attached to the boundary of a domain $\Omega\subset \C^n$. We say that $A_\lambda$ is a testing family for $\mathcal{C}^k$ functions if the following holds:
	If $f\in \mathcal{C}^k(\p\Omega)$ extends holomorphically on $A_\lambda$ for all $\lambda\in\Lambda$ then $f$ is the boundary value of a holomorphic function $F\in \text{Hol}(\Omega)\cap \mathcal{C}^k(\overline{\Omega})$ 
\end{definition}
It is difficult for general domains to prove that a given  family of discs is testing and it is harder to determine the optimal families (i.e. the smallest possible). A model case which can help to understand the problem is when $\Omega$ is the unit ball and the testing discs are chosen among the linear ones. This particular case has been widely studied. We mention here the work of \cite{A11,B13,B16,G12,G12bis,L07,L18} and the most refined version of this case which is in \cite{G12}. In that paper it is considered the family of lines passing through three non-aligned points $a,b,c\in \C^2$ where at least one of the joining lines (say $a$ with $b$) meets the interior of the ball. Such family of lines is a testing family for continuous functions if and only if $a\cdot b \neq 1$ and either $a\cdot c\neq 1$ or $b\cdot c\neq 1$ here $\cdot$ is the Hermitian inner product. Some attempts have been made by the authors to remove the condition on the joining line (see \cite{BP18} and \cite{BF19}) at the price of requiring $f$ to be real analytic. 
We note that while in \cite{BP18}  only two points $a,b$ are needed for real analytic functions if the joining line is tangent to the ball, for lower regularities this is not true as the following counterexample shows
\begin{example}
    The function $f:\mathbb{S}^3 \rightarrow \C$
    \begin{equation}\label{example}
        f(z_1,z_2)=\left(\frac{\bar{z}_1 -\bar{t}_1}{\bar{z}_2 -1}\right)(z_1\bar{t}_2 +z_2-1)\exp\left(\frac{-1}{\sqrt{1-z_2}}\right)
    \end{equation}
    is smooth, extends holomorphically on all lines concurrent to $(t_1,1)$ and $(t_2,1)$. Yet $f$ is not the boundary value of a holomorphic function because it is not CR (see Definition \ref{CR}.) Note that multiplying by extra factors of type $(z_1\bar{t}_k+z_2-1)$ we have a function that extends on lines concurrent to a finite set of points $(t_k,1)$. So the set of lines concurrent to a finite number of points on the same line tangent to the sphere is not a testing family. 
    
\end{example}
In this paper we want to treat the case when one of the joining lines, say $a$ with $b$, is tangent to the unit sphere.          
In this paper we prove that 

\begin{theorem}\label{t1} Let $a,b \in \C^2\setminus \B^2 $ be two distinct points such that the line joining $a$ and $b$ is tangent to the unit sphere $\mathbb{S}^3$ at a point $p$.  
	Let $f:\partial\B^2 \rightarrow \C$ be a continuous function such that $f$ extends holomorphically on every complex line through $a$ and $b$. Then $f$ extends holomorphically on every complex line through $p$.  
\end{theorem}
As a consequence we have the following
\begin{corollary} \label{C1}
	Let $a,b,c \in \C^2\setminus \B^2 $ be three points not on the same complex line and assume that the line joining $a$ and $b$ is tangent to the unit sphere $\mathbb{S}^3$.  
	Let $f:\partial\B^2 \rightarrow \C$ be a continuous function such that $f$ extends holomorphically on every complex line through $a,b$ and $c$. Then $f$ is the boundary value of a holomorphic function $F\in \text{Hol}(\B^2)\cap \mathcal{C}^0(\overline{\B^2})$.
\end{corollary}
The proof is inspired by the one in \cite{G12}. The paper is divided in five sections: In section \ref{S2} we introduce a subgroup of automorphisms with which we perform an averaging procedure on $f$ that allows us to reduce our problem to one complex dimension. In section \ref{S3} we introduce the CR geometry techniques that we need in section \ref{S4} to prove that the "averaged" function is $0$. With this in hand we will prove in section \ref{S5} the Theorem \ref{t1} and from this and \cite{G12} we have Corollary \ref{C1}.   
\section{Reduction to a one dimensional problem }\label{S2}
\subsection{The group $G$ of preserving automorphisms}
It is not restrictive to assume that $a=(t_1,1)$ and $b=(t_2,1)$ with $t_1,t_2 \in \C\setminus \{ 0\}$. 
In order to simplify the problem, we exploit the symmetry of the ball and find the automorphisms that preserve the concurrent lines to $a$ and $b$.
First we recall the general expression of an automorphism of the unit ball:
  \begin{definition}
	For a nonzero vector $a=(a_1,a_2)\in \B^2$ we define the projection over $a$  as $P_a(z):=a\frac{\left\langle z,a\right\rangle }{|a|^2}$ and the orthogonal projection $Q_a(z):=z-P_a(z)$. We also define $s_a=\sqrt{1-|a|^2}$, and for $a=0$, $P_0(z)=0$ and $Q_0(z) =z$.
\end{definition}
From \cite{R}  if $\phi\in \text{Aut}(\B^2)$ then $\phi $ must be of the form 
\begin{equation}\label{aut} 
	\phi(z)=U\left( \frac{a-P_a(z)-s_a Q_a(z)}{1- \left\langle z,a\right\rangle  }\right) 
\end{equation}
for some $a\in\B^2$ and some unitary matrix $U$. Clearly the automorphisms of the ball are meromorphic maps of $\C^2$ and they transform complex lines into complex lines. Because of this we want to find the subgroup $G$ of $\text{Aut}(\B^2)$ of automorphisms that fix the points of the line $z_2=1$. We have

\begin{lemma}\label{l1} The group $G$ is given by
	$$G=\left\lbrace \phi_{a_2}(z)=\frac{1-\bar{a_2}}{1-z_2\bar{a_2}} \begin{pmatrix} z_1 \\ \frac{a_2-z_2}{a_2-1}\end{pmatrix}\ |  \quad a_2=\frac 12(1+e^{i\theta})\ \theta\in [0,2\pi)\right\rbrace .$$
	 Moreover the map 
	$\Phi:\R\rightarrow G $ 
	$$ \Phi (y)=\phi_{\frac{iy}{1+iy}}$$
	is a group homomorphism.
\end{lemma}
\begin{proof}
	We begin by finding those $a\in \B^2$ and $U$ such that $\phi\in G$.
	First we must have by \eqref{aut} that
	\begin{equation}\label{f1}
	 \phi(z_1,1) =U\left( \frac{a-P_a(z)-s_a Q_a(z)}{1- \left\langle z,a\right\rangle  }\right) =\begin{pmatrix} z_1 \\ 1\end{pmatrix} \quad \forall z_1\in \C
	 \end{equation}
	which gives, by the linearity of $P_a$ and $Q_a$, that the denominator of the left-hand side does not depend on $z_1$, thus $a_1=0$.
	We have
	\begin{equation}\label{*}
	\phi(z)=\frac{1}{1-z_2\bar{a}_2} \begin{pmatrix}
	U_{11}&U_{12} \\ U_{21}&U_{22}
	\end{pmatrix}\begin{pmatrix} -\sqrt{1-|a_2|^2} z_1 \\ a_2-z_2 \end{pmatrix}
	\end{equation}
	and again by \eqref{f1} we have that $U$ is a diagonal matrix hence $U_{12}=U_{21}=0$. Moreover, since $U$ is unitary 
	$$\begin{cases}
	-U_{11} \frac{\sqrt{1-|a_2|^2}}{1-\bar{a}_2}=1 &\text{  and } |U_{11}|=1 \\
	U_{22}\frac{a_2-1}{1-\bar{a}_2} =1 &\text{  and } |U_{22}|=1
	\end{cases}.$$
	We find after elementary computations that 
	\begin{equation}\label{**} 
	\begin{cases}
		U_{22}= \frac{1-\bar{a}_2}{a_2-1} \\  U_{11}=\frac{\bar{a}_2 -1}{\sqrt{1-|a_2|^2}} \\
	|a_2|^2-\Re a_2=0
	\end{cases}.
	\end{equation}
	The last equation implies that $a_2$ is on a circle of radius $\frac 12$ and center $\frac 12$ therefore $a_2=\frac 12 (1+e^{i\theta})$.
	Finally, by \eqref{*} and \eqref{**} we have that
	$$ \phi_{a_2}(z):=\frac{1-\bar{a}_2}{1-z_2\bar{a}_2} \begin{pmatrix} z_1 \\ \frac{a_2-z_2}{a_2-1}\end{pmatrix} \in G .$$
	The elements of the group $G$ are in a one to one correspondence with the circle $\mathcal{C}$ of equation $|a_2|^2-\Re a_2=0,\ a_2\neq 1$ and there is a bijection between $\mathcal{C}$ and $\R$ via the following map $\nu :\R \rightarrow \mathcal{C}$, $\nu(y)= \frac{iy}{iy+1}$. We check now that $\Phi (y)=\phi_{\nu(y)}$ is a group homomorphism i.e. we verify that $\Phi(y_1+y_2)=\Phi(y_1)\circ \Phi(y_2)$ or equivalently $\phi_{\nu(y_1+y_2)} =\phi_{\nu(y_1)}\circ \phi_{\nu(y_2)}$.
	We start by considering the fact that given $\phi_a \in G $, then the second component, which is $\frac{(1-\bar{a})(a-z_2)}{(1-z_2\bar{a})(a-1)}$, depends only on $z_2$, it has precisely one zero and such zero coincide with $a$. If we consider two such maps $\phi_a$ and $\phi_b$ in $G$, we have $\phi_b\circ\phi_a\in G$ and therefore there exists $c$ such that $\phi_b\circ\phi_a=\phi_c$. To determine $c$ we just have to find the unique zero of the second component of $\phi_b\circ\phi_a$ which is done by solving
	\begin{equation*}
	 b-\frac{(1-\bar{a})(a-z_2)}{(a-1)(1-z_2\bar{a})}=0 
	 \end{equation*}
	which yields 
	\begin{equation}\label{1bis}
		c=\frac{b(a-1)-a(1-\bar{a})}{\bar{a}b(a-1)-(1-\bar{a})}.
	\end{equation}
	Since $a\in \mathcal{C}$ we have $|a|^2=\Re a$ and $1-|a|^2 =(a-1)(\bar{a}-1)$. With this identity we rewrite \eqref{1bis} and we have
	\begin{equation}\label{2}
	 c=\frac{b(a-1)-a+1+|a|^2-1 }{\bar{a}b(a-1)-1+|a|^2+\bar{a}-|a|^2}=
	 \frac{b-\bar{a}}{\bar{a}b-2\bar{a}+1}.
	 \end{equation}     
	 We choose $a=\frac{iy_1}{1+iy_1}=\nu(y_1)$ and $b=\frac{iy_2}{1+iy_2}=\nu(y_2)$ and replacing into \eqref{2} we have
	 $$ c=\frac{i(y_1+y_2)}{1+i(y_1+y_2)}=\nu(y_1+y_2)$$
and this gives the conclusion $\phi_{\nu(y_1)}\circ\phi_{\nu(y_2)}=\phi_{\nu(y_1+y_2)}$.
\end{proof}
\begin{remark}\label{r1}
	The group $G$ acts on the sphere $\mathbb{S}^3$ in the natural way $\star:G\times \mathbb{S}^3 \rightarrow \mathbb{S}^3$: 
	\begin{equation*}
	 \phi_{\nu(y)}\star (z_1,z_2)=\phi_{\nu(y)}(z_1,z_2)   
	\end{equation*}
	 and this action has a unique fixed point which is $(0,1)$. If $(z_1,z_2)\in\mathbb{S}^3$ then the orbit under the action $\star$ is the boundary of the straight disc passing through $(0,1)$ and $(z_1,z_2)$ deprived of the point $(0,1)$. Note that $\lim_{y\to\infty}\phi_{\nu(y)}(z_1,z_2)=(0,1)$. 
\end{remark} 
The next step is to take the averages of $f$ on the orbits of the action by taking the following integral
$$ \tilde{f}(z_1,z_2):=\int_\R  \!f(\phi_{\nu(y)}(z_1,z_2))\,dy $$
to make sure that the integral converges we multiply $f$ by a holomorphic function vanishing to a high order in $(0,1)$ and non-zero elsewhere. So instead of considering $f$ we take $g(z_1,z_2)f(z_1,z_2)$ where $g(z_1,z_2)=\exp{\left( \frac{-1}{\sqrt{1-z_2}}\right) }$.
\begin{proposition}\label{p1}
	Let $f:\mathbb{S}^3\rightarrow \C$ be a continuous function and let $\tilde{f}$ be the following function:
	\begin{equation}\label{3}
	\tilde{f}(z_1,z_2):=\!\int_\R \!g(\phi_{\nu(y)}(z_1,z_2))f(\phi_{\nu(y)}(z_1,z_2))\,dy
	\end{equation}
	then $\tilde{f}$ is a continuous function on $\mathbb{S}^3\setminus\{ (0,1)\}$ and $G$ invariant: 
	$$ \tilde{f}\circ \phi =\tilde{f}\quad \forall \phi\in G .$$
	Moreover if $f$ extends holomorphically on the family of lines through a point $(t,1),\ t\in\C$, $t\neq 0$ then so does $\tilde{f}$.  
\end{proposition}
\begin{proof}
	First we note that the integral in \eqref{3} converges. To prove that it is enough to check the behavior of the integral for $y\to \infty$. Since
	$$\phi_{\nu(y)}(z_1,z_2)=\left( \frac{z_1}{iy(z_2-1)+1},\frac{iy(z_2-1)+z_2}{iy(z_2-1)+1}\right) $$
	we see that it tends to the point $(0,1)$ for $y$ large with a $\frac{1}{y}$ rate, in fact: 
	$$ \left\| \left(\frac{z_1}{iy(z_2-1)+1},\frac{iy(z_2-1)+z_2}{iy(z_2-1)+1}\right) -(0,1)\right\| \le \frac{|z_1|+|z_2-1|}{|1+iy(z_2-1)|}\le\frac{|z_1|+|z_2-1|}{|y|(1-x_2)}.
	$$ 
	Therefore the integrand functions in \eqref{3} is dominated by $\exp(-|y|^{\frac 12}) $ and so the integral converges
	 and moreover $\tilde{f}$ is continuous in $\mathbb{S}^3 \setminus \{(0,1)\}$. If $\phi_{\nu(y_1)}(z_1,z_2)$ is a point in the $G$-orbit of $(z_1,z_2)$ then, because $\phi_{\nu(y)}\!\circ\!\phi_{\nu(y_1)}=\phi_{\nu(y+y_1)}$, one sees that
	$$ \tilde{f}(\phi_{\nu(y_1)}(z_1,z_2)) =\int_\R \!g(\phi_{\nu(y+y_1)}(z_1,z_2))f(\phi_{\nu(y+y_1)}(z_1,z_2))\,dy=\tilde{f}(z_1,z_2) .$$
	If $A$ is a straight disc through a point $(t,1)$ then $\phi_{\nu(y)}\!\circ A$ is a straight disc through $(t,1)$ (because $\phi_{\nu(y)}$ fixes $(t,1)$ for all $y$ ).
	If $f$ extends holomorphically on every straight disc through $(t,1)$, for one of such discs $A$ and for all $n\in \N$ we have
	
	$$ \int_{\p \Delta}\tau^n \tilde{f}(A(\tau))\,d\tau=\int_{\p\Delta}\! \int_\R \tau^n  fg(\phi_{\nu(y)}(A(\tau))) \,dyd\tau = \int_\R\!\int_{\p\Delta} \tau^n fg(\phi_{\nu(y)}(A(\tau))) \,d\tau dy=0 $$ 
	the last equality holds because $fg$ extends holomorphically  on the disc $\phi_{\nu(y)}(A)$ for all $y$. So $\tilde{f}$ extends holomorphically on $A$.   
\end{proof}
Since by Proposition \ref{p1} $\tilde{f}$ is constant on the $G$-orbits we have that $\tilde{f}$ induces a function on $\mathbb{S}^3/\star$. Since each orbit can be identified with a complex line through $(0,1)$ we can use the "complex slope" $\zeta=\frac{z_1}{z_2-1}$ of this line as a coordinate for $\mathbb{S}^3/\star$. Therefore, we define  $h(\zeta):=\tilde{f}(z_1,z_2)$. Note that $\zeta=0$ corresponds to the line $z_1=0$ and we associate to the line $z_2=1$ the symbol $\zeta=\infty$.
\begin{proposition}\label{p2}
If $\tilde{f}$ extends holomorphically on all lines concurrent to $(t,1)$ then $h$ extends holomorphically on the family of circles centered at $\dfrac{-1}{\bar{t}}$.
\end{proposition}   
\begin{proof}
	Since the map $\zeta(z_1,z_2)=\frac{z_1}{z_2-1}$ is holomorphic, if $\tilde{f}$ extends holomorphically on a disc $A$ then $h$ extends holomorphically on the image disc $\zeta (A)$. We determine now the image under the map $\zeta(z_1,z_2)=\frac{z_1}{z_2-1}$ of the discs through $(t,1)$. 
	Let $v=(v_1,v_2)$ be a unitary complex vector and let $\tau v+(t,1)$ be the complex line through $(t,1)$ and parallel to $v$. The intersection with the unit sphere yields the circle in the $\tau$-plane of equation $|\tau|^2+2\Re(\tau(v_1\bar{t}+v_2)) +|t|^2=0$, which has center $c$ and squared radius $R^2$:  
	\begin{equation}\label{cR}
		c=-(\bar{v_1}t+\bar{v_2}),\ R^2=|\overline{v}_1t+\overline{v}_2|^2-|t|^2
	\end{equation}
	 we call this circle $\mathcal{C}_{c,R}$. We see that to this disc there corresponds a circle in the $\zeta$-plane by
	\begin{equation}\label{4}
	 \zeta(\tau v+ (0,1))=\frac{t+\tau v_1}{\tau v_2}=\frac{v_1}{v_2}+\frac{t}{v_2}\left(\frac{1}{\tau}\right) \quad \tau\in\mathcal{C}_{c,R} .
	 \end{equation} 
	Since the inversion $\frac{1}{\tau}$ transforms a circle of center $c$ and radius $R$ to a circle of center $\frac{\bar{c}}{|c|^2-R^2}$ and radius $\frac{R}{|R^2-|c|^2|}$ by \eqref{cR} we have that \eqref{4} gives a circle in the $\zeta$-plane of center $\dfrac{-1}{\bar{t}}$ and radius $\sqrt{\frac{|v_2|^2(1-|t|^2)+2\Re(t\overline{v}_1v_2)}{|v_2t|^2}}$. Clearly if $v_2$ is small the resulting radius tends to $\infty$ therefore by continuity we have all circles centered at $\dfrac{-1}{\bar{t}}$.  
\end{proof}
\subsection{Behavior of $h$ at infinity}
Clearly $\tilde{f}$ is continuous on $\mathbb{S}^3\setminus (0,1)$ and $\tilde{f}(0,1)=0$ because when $z_1=0, z_2=1$ the integrand function in \eqref{3}  is identically zero. Since $\tilde{f}$ is constant on the $G$-orbits it cannot be continuous in $(0,1)$ unless it is constant. We can prove that the induced function $h$ is continuous and infinitesimal at $\infty$ 
\begin{proposition}\label{p3}
	If $f:\mathbb{S}^3\rightarrow\C$ is continuous then for the corresponding function $h$ we have $h(\zeta)=O_\infty(|\zeta| e^{-|\zeta|})$ .
\end{proposition} 
\begin{proof}
 Since $\zeta=\frac{z_1}{z_2-1}$ we have to prove that  $h(\frac{z_1}{z_2-1})=\tilde{f}(z_1,z_2)$ is small when $|\frac{z_1}{z_2-1}|$ is large. We begin with  
 \begin{equation}\label{5}|\tilde{f}(z_1,z_2)|\le\int_\R |gf(\phi_{\nu(y)})|\,dy\le C\int_\R \left| \exp\left( -\sqrt{\frac{iy(z_2-1)+1}{1-z_2}}\right) \right|\,dy \end{equation}
 where $C=\max{f}$. If we put $z_2=x_2+iy_2$ the last integral in \eqref{5} is equal to
 \begin{equation}
 \int_\R\left| \exp\left( -\sqrt{\frac{1-x_2+iy_2}{(1-x_2)^2+y^2_2}-iy}\right)   \right| \,dy =\int_\R\left| \exp\left( -\sqrt{\underbrace{\frac{1-x_2}{(1-x_2)^2+y^2_2}}_{=\alpha(z_2)}+iy}\right)   \right| \,dy
 \end{equation}
 To estimate this last integral we consider the real part of the complex square root inside the exponential (for short we write $\alpha$ instead of $\alpha(z_2)$):
 $$\Re \sqrt{\alpha+iy}=(\alpha^2+y^2)^{\frac{1}{4}}\cos\frac{\arg(\alpha+iy)}2 =\frac{1}{\sqrt{2}}\sqrt{\alpha+\sqrt{\alpha^2+y^2}} $$
 so the integral becomes, after another change of variable  
 \begin{align} \int_\R \alpha&\exp\left( -\sqrt{\frac{\alpha}{2}}\sqrt{1+\sqrt{1+\left(\frac{y}{\alpha}\right)^2}}\right) \,d\left( \frac{y}{\alpha}\right)=\int_\R \alpha\exp\left( -\sqrt{\frac{\alpha}{2}}\sqrt{1+\sqrt{1+y^2}}\right) \,dy \nonumber \\
 &\le \int_\R \alpha\exp\left( -\sqrt{\frac{\alpha}{2}}\sqrt{1+|y|}\right) \,dy\le 
 \int_\R \alpha\exp \left( - \frac{\sqrt{\alpha}}{2}(1+\sqrt{|y|})\right) \,dy \nonumber \\ \label{8}
 &\le \alpha e^{-\frac{\sqrt{\alpha}}{2}}\int_\R \exp \left( -\frac{\sqrt{\alpha} |y|}{2}\right) \,dy
 \end{align}
 and if $\alpha$ is large then the integral is uniformly small.
 If $(z_1,z_2)\in \mathbb{S}^3$ is such that $|\frac{z_1}{z_2-1}|>M$ we have
 \begin{equation} \begin{cases} \label{6}
 |z_1|^2+|z_2|^2 =1 \\
 |z_1|^2>M^2|z_2-1|^2
 \end{cases} 
 \end{equation} 
 from \eqref{6} follows $(M^2+1)|z_2|^2 -2M^2x_2+M^2-1<0$ and this implies
 \begin{equation} \label{7}\frac{1-x_2}{|z_2-1|^2} >\frac{M^2+1}2 \end{equation}
 which is $\alpha(z_2)>\frac{M^2+1}{2}$.
 Therefore, for $|\zeta|>M$, let $(z_1,z_2)$ be such that $\zeta =\frac{z_1}{z_2-1}$, by \eqref{7} and \eqref{8} we have
 \begin{equation}\label{21bis} h(\zeta)=\tilde{f}(z_1,z_2)=O_\infty(M e^{-M}) .\end{equation}
 \end{proof}
 \section{CR Geometry}\label{S3}
 We recall here some definitions about the Cauchy-Riemann geometry (in short CR) that we shall need. This language is important when describing the structure and the properties of real submanifolds in a complex space. First we recall the definition of the complex structure $J$: 
 \begin{definition}
  In $\C^n$ with coordinates $z=(z_1,...,z_n)=(x_1+iy_1,...,x_n+iy_n)$ the standard complex structure is the linear map $J$ defined on the tangent bundle $J:T\C^n\rightarrow T\C^n$ such that $J(\p_{x_h})=\p_{y_h}$ and $J(\p_{y_h})=-\p_{x_h}$ for $1\le h\le n$.
 \end{definition}
   Since $J$ is an anti involution, $J^2=-Id$, we have that $J$ has two complex eigenvalues $\pm i$. This induces a decompositions on $\C\otimes T_z\C^n \simeq T^{1,0}_z\C^n \oplus T^{0,1}_z\C^n$ of holomorphic and anti-holomophic vector bundles.
 \begin{definition}
 	We define $$ T^{1,0}_z\C^n=\{\sum_{i=1}^{n}a_i \p_{z_i} |\ a_i\in\C\}$$
 	and similarly
 	$$ T^{0,1}_z\C^n=\{\sum_{i=1}^{n}a_i \p_{\bar{z}_i} |\ a_i\in\C\}.$$
 \end{definition}  
 \begin{definition}
     Let $M\subset \C^n$ be a real submanifold, the complex tangent space of $M$ at a point $z$ is
     $$ T^\C_z M:= T_zM\cap JT_z M$$
     which is the largest complex space contained in $T_zM$.
     We say that $M$ is a CR manifold if $T^\C_z M$ has constant dimension for $z\in M$. In this case $\dim_\C T^\C M $ is called the CR-dimension of $M$. If the CR-dimension of $M$ is $0$ then we say that $M$ is totally real.
     
     We say that $M$ is generic if $T_zM+JT_zM=T\C_z^n$, and we say that $M$ is maximally totally real if it is totally real and generic. 
 \end{definition}
 We want to introduce the anti-holomorphic tangent bundle of a CR manifold: 
 this is the bundle of anti-holomorphic vectors that are tangent to $M$
 $$T^{0,1}_zM := (\C\otimes TM)\cap T^{0,1}_z\C^n .$$
 \begin{example}
 	We recall here that manifolds of co-dimension $1$ are generic CR manifolds and that $\R^n\hookrightarrow \C^n$ is a maximally totally real manifold.
 \end{example}
Analogously, we define $T^{1,0}M:=\C\otimes TM\cap T^{1,0}\C^n$. It is easy to check that if $M$ is a CR manifold then $T^{0,1}M$ is a bundle whose complex dimension is equal to the CR dimension of $M$ and moreover we have $\C\otimes T^\C M =T^{0,1}M\oplus T^{1,0}M$. 
 We define now the CR functions which are a natural generalization of holomorphic functions:     
 \begin{definition}\label{CR}
          A $\mathcal{C}^1$ function on a CR manifold $M$ $f:M\rightarrow \C$  is said to be CR if for all vector fields $X\in T^{0,1}M$ we have $Xf=0$.
 \end{definition}
The above definition can be easily extended in a weak sense, using distributions, to continuous functions. All results on CR functions that we shall use hold for this weaker notion. If $M$ is a CR manifold with boundary, $N$ we say that $f:M\rightarrow \C$ is CR if it is continuous and CR in the interior of $M$. 
Restrictions and boundary values of holomorphic functions are CR functions, in contrast not all CR functions are restrictions of holomorphic ones. Whether a CR function has this property has been widely studied, and it has been proved that a crucial role is played by $T^\C M$ and by its integral manifolds. In general the bundle $T^\C M$ is not integrable, so we give the following
\begin{definition}
	A CR manifold $M$ is called Levi-flat if the bundle $T^\C M$ is integrable. The integral leaves of $T^\C M$, which are complex manifolds of dimension equal to the CR dimension of $M$, are called Levi-leaves. 
\end{definition}     
A continuous function on a Levi-flat manifold is CR if and only if its restrictions on the Levi-leaves are holomorphic.
Levi-leaves play an important role in propagation of regularity, but we will come back later to this point when we need it.
We will need the following tools from CR geometry and for our considerations all manifolds will be assumed smooth.
\begin{definition}\label{def}
	We say that a CR manifold $M'$ is attached to $M$ at a point $z\in M$ in direction $X\in T_z\C^n$ if, in a neighborhood of $z$, $M'$ is a manifold with boundary $M$, $X$ is tangent to $M'$ and $X$ points inside $M'$.
	A CR function $f$ on $M$ extends CR at $z\in M$ in a direction $X$ if there exists a CR manifold $M'$ attached to $M$ in direction $X$ and a continuous CR function $F$ on $M'$ such that $F|_M=f$.  
\end{definition}
When $M$ is a hyper-surface, a manifold $M'$ attached to $M$ corresponds to one of the two components in which the ambient space is divided by $M$. In this case, the vector $X$ points inside the component.
The definition \ref{def} can be extended to CR manifolds with generic boundary. 
\begin{definition}
	Let $M$ be a CR manifold with boundary $N$, $z\in N$ a point, $X\in T_z\C^n \setminus T_zM$ a tangent vector. We say that a manifold $M'$ is attached to $M\cup N$ in direction $X$ if $M'$ is a CR manifold that locally around $z$ is a manifold with boundary $N$ such that $X$ is tangent to $M'$ and points inside $M'$.
\end{definition}
This definition has been introduced in \cite{T97} to describe the propagation of CR extension on CR manifolds with boundary. It turns out that CR curves, i.e. curves whose velocity is in $T^\C M$, are propagators of CR extension. 
If $\gamma :[0,1]\rightarrow M$ is a curve with end-points $p_0=\gamma(0),\ p_1=\gamma(1)$, such that $\dot{\gamma}\in T^\C M$ and if $f$ CR-extends at $p_0$ in direction $X_0$ then $f$ CR-extends at $p_1$ to a direction $X_1$ which is related to $X_0$.
The relation between $X_0$ and $X_1$ is strong and evident on the components complementary to a certain sub-space. Namely, let $S\subset M$ be the smallest manifold which contains all the CR curves through $p_0$, the so-called CR orbit of $p_0$. We have $\dim_{CR} S=\dim_{CR}M$, if $p_0\in S$ then the entire curve $\gamma \subset S$. 
We introduce the following bundle
\begin{definition}
	Let $M$ be a CR manifold with generic boundary $N$, $S\subset M$ a CR manifold with boundary $S_0\subset N$ such that $\dim_{CR}M=\dim_{CR}S$.
	If $z\in M$ we define 
	\begin{equation} \label{bundle}
	\mathscr{E}_z:= \frac{T_z\C^n}{T_zM+JT_zS} 
\end{equation}
	and for $z\in N$, 
\begin{equation*}
	\mathscr{E}_z:= \frac{T_z\C^n}{T_zN  +JT_z S_0}.
\end{equation*}

\end{definition}
Note that since $N$ is generic we have, on points of $N$, $TN+JTS_0=TM+JTS$ so we can keep \eqref{bundle} as a definition for $\mathscr{E}$ at boundary points. It turns out there exists a partial connection and a parallel displacement $\Pi$ on this bundle which relates the classes of $X_0$ and $X_1$ in $\mathscr{E}$. 
\begin{theorem}[Tumanov] \label{tumanov}
Let $M$ be a CR manifold with generic boundary $N$ and let $S\subset M$ be a CR submanifold of $M$ with boundary $S_0 \subset N$ such that $\dim_{CR}M=\dim_{CR}S$. Let $p_0$ and $p_1$ be two points in $M\cup N$ joined by a CR curve $\gamma :[0,1]\rightarrow M\cup N$. Suppose that $f$ CR-extends at $p_0$ in direction $X_0$ then there exists for every $\epsilon>0$ a manifold $M''$ attached to $M\cup N$ at $p_1$ in direction $X_1$ such that 
\begin{equation}
	\left\|\Pi_\gamma ([X_0]) -[X_1]\right\|_{\mathscr{E}_{p_1}}\le \epsilon 
\end{equation}
and $f$ CR-extends to $M''$.      	
\end{theorem}
\begin{remark}
	If $M$ in Theorem \ref{tumanov} is a Levi-flat hyper-surface we have that $S$ is a Levi-leaf, therefore $TS=T^\C M =JTS$ which in turn implies $\mathscr{E}_z=\frac{T_z\C^n}{T_zM}=\mathscr{N}_z(M)$ the normal bundle which has real dimension 1. Moreover, since $S$ is complex every curve $\gamma$ in $S$ is CR. The description of $\Pi$ in \cite{T97} is possible, and easier, by means of its dual $\Pi^*$ which is realized in the following way:
	Since $\mathscr{E}_z^*=\mathscr{N}_z^*(M)$ this space can be identified with the holomorphic $1$-forms whose real part vanishes on $T_zM$ . This bundle is generated by $\p r$ where $r$ is an equation of $M$. Since $M$ is Levi-flat it can be shown that $\mathscr{E}^*$ is itself a CR manifold in $T^*\C^n$ and 
	 that for any CR curve $\gamma$ there exists a lift $\gamma^* (t)=(\gamma(t),\lambda(t)\p r(\gamma(t)))$ which is a CR curve in $\mathscr{E}^*$. Thanks to this, we can express a duality relation between $X_0$ and $\Pi([X_0])$ 
	 \begin{equation}\label{duality}
	 	\left\langle \gamma^*(1),\Pi_\gamma([X_0])\right\rangle =c(\gamma)\left\langle \gamma^*(0),X_0\right\rangle 
	 \end{equation}
 where $c(\gamma)$ is a positive constant    
	(for details see \cite{T97} and the references therein).
\end{remark}
\begin{remark}\label{remark2}
	If $M$ is a Levi-flat CR hyper-surface in $\C^2$ whose equation is $r=\Re( g(z_1,z_2))=0$ where $g$ is a holomorphic function, then the Levi-leaves are the complex curves   $g(z_1,z_2)=iC$ for $C\in \R$. Every leaf has a holomorphic lift to $\mathscr{N}^*(M)$ which is given by $(z_1,z_2;\p r(z_1,z_2))$ because $\p \Re g(z_1,z_2)=\p_{z_1} g dz_1 +\p_{z_2}g dz_2$ is holomorphic. Therefore, $\mathscr{N}^*(M)$ is foliated by complex curves. Let $\gamma$ be a CR curve in $M$, this will be contained in a Levi-leaf, the lift $\gamma^*$ can be easily computed:  $\gamma^*(t)=(\gamma(t),\p g(\gamma(t)))$.
\end{remark}
\section{Proof of main theorem}\label{S4}
 In this section we consider the following complexification of $\C$:
 \begin{definition}
     Let $\C$ be the complex plane with coordinate $\zeta$, we define the map
    
         \begin{align*}
             \iota:&\C \rightarrow \C^2 \\
             &\zeta\rightarrow (\zeta,\bar{\zeta})
         \end{align*}
which embeds  the complex plane $\C$ into the maximal totally real set $\triangle:=\left\{ (\zeta,\bar{\zeta}),\ \zeta\in\C\right\}$ .         
\end{definition}
  In the complex plane for every circle $\mathcal{C}_{c,R}$  of equation $|\zeta-c|^2=R^2$ we consider its complexification in $\C^2$ which is obtained by replacing $\bar{\zeta}$ with $\eta$ in the equation $(\zeta-c)(\eta-\bar{c})=R^2$. The intersection of this quadric with $\triangle$ is $\iota(\mathcal{C}_{c,R})$. We identify the disc $\D_{c,R}$, whose boundary is $\mathcal{C}_{c,R}$, with the part of the quadric that projects over it which means that for $0<|\zeta-c|\le R $ we take $\eta=\bar{c}+\frac{R^2}{\zeta-c}$. Note that in this procedure the center of the disc is sent to $\infty$ and the boundary circle $\mathcal{C}_{c,R}$ is sent to $\iota(\mathcal{C}_{c,R})$. We denote by 
 $$ \overline{\D}_{c,R}^\C :=\left\{  \left(\zeta,\bar{c}+\frac{R^2}{\zeta-c}\right) \quad :\ 0<|\zeta-c|\le R \right\} $$
 If $h:\C\rightarrow \C$ extends holomorphically from the circle $\mathcal{C}_{c,R}$ then we define the extension on $\D_{c,R}^\C$ in the following way  $$\tilde{h}(\zeta,\eta)=\frac{1}{2\pi i}\int_{\p\D_{c,R}} \frac{h(w)}{w-\zeta}\, dw \quad \text{ for }(\zeta,\eta)\in \overline{\D}^\C_{c,R}.$$
 Note that $\tilde{h}(\zeta,\bar{\zeta})=h(\zeta)$ for $\zeta\in \mathcal{C}_{c,R}$. Since the function $h$ that we are considering extends holomorphically on families of concentric circles we introduce the following set 
 \begin{definition}
     Let $c\in \C$ we define
     \begin{equation} \label{mc}
     M_c:= \bigcup_{R\ge 0} \overline{\D}_{c,R}^\C 
     \end{equation}  
 \end{definition}
 \begin{remark}
      If $R_1\neq R_2$ then it is easy to check that $\overline{\D}_{c,R_1}^\C\cap \overline{\D}_{c,R_2}^\C=\emptyset$. Therefore if $h$ extends holomorphically on the family of circles centered in $c$ then the extension $\tilde{h}$ is well defined on $M_c$. 
     We note that two circles with different centers may have non-trivial intersection. Let $\mathcal{C}_{c_1,R_1},\mathcal{C}_{c_2,R_2}$ be two circles then $\overline{\D}_{c_1,R_1}^\C \cap \overline{\D}_{c_2,R_2}^\C $ contains at most $2$ points, in fact taking the intersection of the quadrics
     \begin{equation}\label{10}
         \begin{cases}
           (\zeta-c_1)(\eta-\bar{c}_1)=R^2_1 \\
           (\zeta-c_2)(\eta-\bar{c}_2)=R^2_2
         \end{cases}
     \end{equation}
     we have that $\zeta$ has to satisfy the following equation
     \begin{equation}\label{11}
        (\bar{c}_1-\bar{c}_2)\zeta^2+(R_1^2-R^2_2-(\bar{c}_1-\bar{c}_2)(c_1+c_2))\zeta+(\bar{c}_1-\bar{c}_2)c_1c_2+R^2_2c_1-R^2_1c_2=0 
        \end{equation}
        which has only two solutions (if $c_1\neq c_2$). If the circles $\mathcal{C}_{c_1,R_1},\mathcal{C}_{c_2,R_2}$ intersect at $\zeta_1,\zeta_2$ then the solutions of \eqref{10} will be $(\zeta_i,\bar{\zeta}_i),\ i=1,2$. In order for $\overline{\D}_{c_1,R_1}^\C,\overline{\D}_{c_2,R_2}^\C$ to have non-empty intersection outside $\triangle$, the corresponding solutions of \eqref{11} must lay inside the circles bounded by $\mathcal{C}_{c_1,R_1},\mathcal{C}_{c_2,R_2}$ and this happens if and only if one of the two circles surrounds the other (see \cite{T07, G12}).   
 \end{remark}
\begin{proposition}
    The set $M_c\setminus \{ (c,\bar{c})\}$ is a Levi-flat hyper-surface with boundary $\triangle \setminus \{(c,\bar{c})\}$. 
\end{proposition}
\begin{proof}
We first note that $M_c$ is contained in the subset of $(\zeta,\eta)\in \C^2$ such that $(\zeta-c)(\eta-\bar{c})=R^2$ for some $R$. This in turn yields
\begin{equation}
	M_c\subset \{ (\zeta,\eta) \in \C^2 |\  \Im \left( (\zeta-c)(\eta-\bar{c})\right)=0 \}
\end{equation}

and the right-hand side is a Levi flat hyper-surface in $\C^2$ except at the singular point $(c,\bar{c})$. Therefore, where $M_c$ is a manifold it is Levi-flat. A point $(\zeta,\eta)$ is in $M_c$ if and only if \begin{equation}\label{9}
    \begin{cases}
      (\zeta-c)(\eta-\bar{c})\text{ is real and non-negative} \\
     |\zeta-c|^2\le(\zeta-c)(\eta-\bar{c}). 
    \end{cases}
\end{equation}
 Where the second inequality is strict, we have that $M_c$ is charted by the following map 
\begin{equation*}
    (\zeta,R)\mapsto (\zeta,\bar{c} +\frac{R^2}{\zeta-c})
\end{equation*}
which has rank $3$ therefore at those points $M_c$ is a manifold. The boundary of $M_c$ is given when equality holds in the second equation of \eqref{9}, this happens only at points where $\eta=\bar{\zeta}$ which is $\triangle \setminus (c,\bar{c})$.
\end{proof}
In the sequel, we will adopt the following convention: for $c\in \C$ we set $r_c=\Im \left( (\zeta-c)(\eta-\bar{c})\right)$ and $g_c(\zeta,\eta)=(\zeta-c)(\eta-\bar{c})$ so that at regular points a local equation for $M_c$ is $r_c=0$.
In the next proposition we want to study the intersection of two of such manifolds. It is not restrictive to consider the intersection of $M_0$ and $M_1$ because the general case can be easily reduced to this one. 
\begin{proposition} \label{p5}
    The intersection $M_0\cap M_1 =\triangle\cup T$ where $T=\{ (\zeta,\eta)\in \R^2\ |\ (\zeta\le 0,\  \eta\le \zeta  )\text{ or } (\zeta\ge 1,\  \eta\ge \zeta)\}$. The two manifolds are transverse at $\triangle$ except along the points $(t,t)$ for $t\in\R$. 
\end{proposition}
\begin{proof}
We begin by computing the intersection of $M_0$ and $M_1$. Every point of the intersection can be found as the intersection of two of the quadrics that form $M_0$ and $M_1$:  
\begin{equation}\label{13}
    \begin{cases}
      \zeta\eta=R^2_1 \\
      (\zeta-1)(\eta-1)=R^2_2 \\
      |\zeta|\le R_1 \\
      |\zeta -1|\le R_2
    \end{cases}
\end{equation}
after replacing $\zeta\eta$ in the second equation we have that $-\zeta-\eta =R^2_2-R^2_1-1$ which implies that $\zeta=t+iu$ and $\eta=s-iu$ have opposite imaginary parts. Replacing in the first equation, since $\zeta\eta$ has to be real, we have $(t-s)u=0$. We have either
\begin{equation}\label{19}
    t-s=0 \text{ which implies } \eta=\bar{\zeta} 
\end{equation}
or 
\begin{equation} \label{20}
 u=0 \text{ which forces } \zeta, \eta\in\R.
\end{equation}
If \eqref{19} holds then we get the points of $\triangle$ that we know belong to the boundary of $M_0$ and $M_1$. 
If $\zeta,\eta\in\R$ we have by \eqref{13} that
$\zeta^2+(R^2_2-R^2_1-1)\zeta+R^2_1=0$. From this equation we see that for only one of its roots $|\zeta|\le R_1$. To have real solutions, it must be $(R^2_2-R^2_1-1)^2\ge 4R^2_1$. If $R^2_2-R^2_1-1\ge0$ we find $R^2_2-R^2_1-1-2R_1\ge 0$ from which follows 
\begin{equation}\label{14}
R_2\ge R_1+1.    
\end{equation}
If instead $R^2_2-R^2_1-1<0$ it must be $(R_1-1)^2>R^2_2$ from which follows
\begin{equation}\label{15}
    R_1-1>R_2 \text{ or } R_1+R_2<1
\end{equation}
If \eqref{14} holds, then we find that
\begin{equation}
    \zeta= \frac{R_1^2+1-R^2_2+\sqrt{(R_1^2+1-R^2_2)^2-4R_1^2}}{2},\eta=\frac{R_1^2+1-R^2_2-\sqrt{(R_1^2+1-R^2_2)^2-4R_1^2}}{2}
\end{equation}
And this solution also satisfies the last inequality in \eqref{13}.
Collecting all solutions obtained in this way by varying $R_1,R_2$ we get the region $\zeta>\eta, \zeta<0$.
If \eqref{15} holds, the solution to \eqref{13} is
\begin{equation}
    \zeta= \frac{R_1^2+1-R^2_2-\sqrt{(R_1^2+1-R^2_2)^2-4R_1^2}}{2},\eta=\frac{R_1^2+1-R^2_2+\sqrt{(R_1^2+1-R^2_2)^2-4R_1^2}}{2}
\end{equation}
while for $R_1+R_2<1$ there are no solutions. Collecting all points of this kind, we have the region $\zeta>1, \zeta<\eta$. 
The interior of $M_0$ has equation $r_0:\frac{\zeta\eta-\overline{\zeta\eta}}i=0$ whose complex differential is $\partial r_0=\frac{\eta}i \,d\zeta+\frac{\zeta}i\,d\eta$ and similarly for $M_1$ we have $r_1:\frac{(\zeta-1)(\eta-1)-\overline{(\zeta-1)(\eta-1)}}i=0$
 and $\partial r_1 =\frac{\eta-1}i \,d\zeta+\frac{\zeta-1}i\,d\eta$.
 At points of type $(\zeta,\bar\zeta)\in \triangle$ the tangent space of $M_0$ is $T_{(\zeta,\bar\zeta)}M_0 =T_{(\zeta,\bar\zeta)}\triangle\oplus \langle X\rangle $ where $X$ is a complex tangential vector to $M_0$ at $(\zeta,\bar\zeta)$. 
 We have that $\Re \p r_0$ vanishes on $TM_0$ and $\p r_0$ vanishes on $X$. 
 Similarly, for $M_1$ we have $T_{(\zeta,\bar\zeta)}M_1=T_{(\zeta,\bar{\zeta})}\triangle\oplus\langle Y\rangle$ where $Y$ is a complex tangent vector to $M_1$ at $(\zeta,\bar{\zeta})$ so if $\p r_0$ and $\p r_1$ are independent at  $(\zeta,\bar\zeta)$ we have that $TM_0$ and $TM_1$ are transverse.
 This happens if $\left|\begin{array}{cc} 
     \eta & \zeta \\
     \eta-1 &\zeta-1 
 \end{array}\right| =-\eta+\zeta\neq 0$. At points $(\zeta,\bar\zeta)$ where $\Im\zeta \neq 0$, $M_0$ and $M_1$ are transverse. 
If $\zeta_0\in\R$ we have that in $(\zeta_0,\zeta_0)$ $M_0$ and $M_1$ have the same tangent space or better consider the two circles $\mathcal{C}_{0,|\zeta_0|}$ and $\mathcal{C}_{1,|\zeta_0 -1|} $.
The complex quadric $\mathcal{C}_{0,|\zeta_0|}^\C$ has equation
\begin{equation}
    \eta=\frac{\zeta_0^2}{\zeta},\  |\zeta|\le|\zeta_0| 
\end{equation}
and we see that the complex tangent direction to $M_0$ is given by $(-\zeta_0,\zeta_0)$ and similarly for $M_1$ a tangent direction is given by $(-\zeta_0+1,\zeta_0-1)$. We notice that for $\zeta_0<0$ or $\zeta_0>1$ the two manifolds lay on the same side of the edge. If instead $0<\zeta<1$ then $M_0$ and $M_1$ are opposite. 
\end{proof}
We have this proposition on $h$:
\begin{proposition}\label{p4}
    Let $c_1,c_2\in \C$ with $c_1\neq c_2$, $h:\C\rightarrow \C$ be a continuous function such that: $h$ extends holomorphically on every circle centered at $c_1$ or $c_2$, and $ h(\zeta)=O_\infty(|\zeta| e^{-|\zeta|})$. Then $h(\zeta)=0$ for all $\zeta$. 
    
\end{proposition}

\begin{proof}
The proof of this proposition is divided in $3$ steps which aim to prove that $h$ is holomorphic and hence $0$ by Liouville.
It is not restrictive to assume that $c_1=0$ and $c_2=1$. 
\subsubsection*{Step 1. Extension of $h$}
Since $h$ extends holomorphically on all discs centered at $0$ and $1$, we have that $h$ defines an extension $\tilde{h}_0$ on $M_0$ and another one $\tilde{h}_ 1$ on $M_1$. Since $M_0$ and $M_1$ have nontrivial intersections, further analysis is needed for these points. 
Inside $M_0\cap M_1$ there is $\triangle \setminus \{(0,0),(1,1)\}$ and at these points the two extensions match because for both $\tilde{h}(\zeta,\bar{\zeta})=h(\zeta)$.
The other points outside the diagonal will be treated in Step 3 of the proof. Outside the intersection, a continuous function $\tilde{h} : \left((M_0\cup M_1)\setminus (M_0\cap M_1)\right)\cup\triangle \rightarrow \C $ is well-defined, and moreover is $CR$.
We note that at points $(\zeta_0,\bar{\zeta}_0)$ with $\Im\zeta_0\neq 0$
the function $\tilde{h}|_\triangle$ extends to be CR on the two transverse manifolds $M_0$ and $M_1$ by Proposition \ref{p5}. Consider the complex curve $\overline{\D}^\C_{0,|\zeta_0|}\subset M_0$ which is the CR orbit of $(\zeta_0,\bar{\zeta}_0)$ inside $M_0$.
We have that $\tilde{h}$ is a CR function on $M_0$ and we have that $\tilde{h}$ CR-extends at $(\zeta_0,\bar{\zeta}_0)$ in direction $Y(\zeta_0)=(-\zeta_0+1,\bar{\zeta}_0-1)$.
By \cite{T97} this extension propagates on the points of $\overline{\D}^\C_{0,|\zeta_0|}$ in $M_0$. This means, since $M_0$ is a hypersurface in $\C^2$, that for all points $(\zeta,\eta)\in \overline{\D}^\C_{0,|\zeta_0|}$ there exists a neighborhood $U_{(\zeta,\eta)}$ and a holomorphic function $\tilde{H}$ defined on one of the two components of $U_{(\zeta,\eta)}\setminus M_0$, continuous up to the boundary and such that $\tilde{H}|_{U_{(\zeta,\eta)}\cap M_0}=\tilde{h}$. The side of extension depends continuously on the initial direction $Y(\zeta_0)$ at $(\zeta_0,\bar{\zeta}_0)$. Let $\gamma$ be a real curve inside $\overline{\D}^\C_{0,|\zeta_0}$ connecting $(\zeta_0,\bar{\zeta}_0)$ to $(\zeta,\eta)$, by Remark \ref{remark2} and by \eqref{duality} if $\p r_0$ is a non-vanishing conormal of $M_0$, the sign of $\Re\langle \p r_0 (\zeta,\eta), \Pi[Y_{\zeta_0}]\rangle$ is the same as $\Re \langle \p r_0(\zeta_0),Y(\zeta_0)\rangle$. We can repeat the same procedure by choosing a different initial point on $\overline{\D}^\C_{0,|\zeta_0|} \cap \triangle$, in this case we choose $(\zeta_0 e^{i\theta},\bar{\zeta}_0e^{-i\theta})$, and we have $Y(\zeta_0 e^{i\theta}) = (-\zeta_0 e^{i\theta}+1,\bar{\zeta}_0 e^{-i\theta}-1) $ and
$\p r_0 (\zeta_0 e^{i\theta},\bar{\zeta}_0e^{-i\theta})=\frac{1}{i}(\bar{\zeta}_0e^{-i\theta},\zeta_0 e^{i\theta})$. If we take then $\Re \langle \p r_0,Y \rangle = \frac{\bar{\zeta}_0e^{-i\theta} - \zeta_0 e^{i\theta}}{i}$ we see that for different $\theta$ it assumes values of both signs which means that at every point $(\zeta,\eta)\in \overline{\D}^\C_{0,|\zeta_0|}$ the function $\tilde{h}$ extends to both sides of $M_0$ thus $\tilde{h}$ extends holomorphically to a neighborhood of $(\zeta,\eta)$. Since we can repeat the same argument for all circles of the family, we get that $\tilde{h}$ extends holomorphically to a neighborhood of $(M_0\cup M_1)\setminus (M_0\cap M_1)$. 
\subsubsection*{Step 2. Analytic continuation to $M_0\cap M_1$.} 
We notice that $\tilde{h}_0$ extends holomorphically in a neighborhood of $M_0\setminus (M_0\cap M_1)$, using again propagation we have that $\tilde{h}_0$ extends holomorphically to a neighborhood of $M_0\setminus \triangle$. We note that near a point of  $\triangle \setminus \{(0,0)\}$ the set $M_0$ is bi-holomorphically equivalent to a half hyper-plane in $\C^2$.  By a change of coordinates 
\begin{equation*}
\begin{cases}
w_1=\zeta\eta \\
w_2=i\log\left(\frac{\eta}{\zeta}\right)
\end{cases}
\end{equation*}
we have that $M_0$ is equivalent to $\Im( w_1)=0, \Im( w_2) \ge 0$. By the classical "local Bochner's tube" 
 (see \cite{K}, \cite{T97}) we have that $\tilde{h}_0$ extends holomorphically to a wedge $\mathcal{W}_0$ with edge $\triangle\setminus (0,0)$. Note that $M_0\setminus \triangle$ is contained in the interior of $\mathcal{W}_0$. Similarly, the same holds for $\tilde{h}_1$ which extends  holomorphically to a neighborhood of $M_1\setminus \triangle$ and near the boundary to a wedge $\mathcal{W}_1$ with edge $\triangle \setminus \{(1,1)\}$.   
By Proposition \ref{p5} $M_0\cap M_1 =\triangle\cup T$ is a connected set, and $\tilde{h}_0=\tilde{h}_1$ on $\triangle$. The other points of the intersection  $(\zeta,\eta)\in T\subset \R^2$ are in the interior of $M_0$ and $M_1$, and on these points we have to prove that these two extensions match. Let $(\zeta,\eta)\in T$ which means that either $\zeta>1$ and  $\eta>\zeta$ or $\zeta<0$ and $\eta<\zeta$. Let us take a point of the first kind, and consider the curve $\gamma_{(\zeta,\eta)}(t)=(\eta,(1-t)\zeta +t\eta)$ for $0\le t\le 1$ which joins the point $(\zeta,\eta)$ to the edge $\triangle$. 
We note that in a neighborhood of the point $(\eta,\eta)\in\triangle$ the function $\tilde{h}$ extends to a holomorphic function in two wedges $\mathcal{W}_0$ and $\mathcal{W}_1$.
Since $\gamma_{(\zeta,\eta)}(t)\in \mathcal{W}_0 \cap \mathcal{W}_1$ for $t$ close to $1$ this shows that there is a wedge $\mathcal{V}$ with edge a small neighborhood of $(\eta,\eta)$ in $\triangle$ and contained in $\mathcal{W}_0 \cap \mathcal{W}_1$. It follows that on $\mathcal{V}$ the functions $\tilde{h}_0$ and $\tilde{h}_1$ are holomorphic and have the same value at the edge, hence they must coincide on $\mathcal{V}$. By analytic continuation $\tilde{h}_0 =\tilde{h}_1$ on the whole connected component of $T$ that contains $(\zeta,\eta)$. This proves that $\tilde{h}_0$ and $\tilde{h}_1$ coincide  on $T$. We call  again $\tilde{h}$ the so defined function on $M_0\cup M_1$.
\subsubsection*{Step 3. Extension of $\tilde{h}$ outside $M_0\cup M_1$} In this part of the proof we aim to prove that $\tilde{h}$ really depends only on $\zeta$ and hence that $h$ is holomorphic. To this end, we follow the idea of \cite{T07} and exploit the Hans Lewy's technique: For every $\zeta$ and $j=0,1$ let 
$$ E^\zeta_j=\{\eta\in \C | (\zeta,\eta)\in M_j \} $$
and $E^\zeta:= E^\zeta_0\cup E^\zeta_1$.
It is easy to see that $E_\zeta$ is made of two half-lines issued from $\bar{\zeta}$.  We have
\begin{equation}
        (\zeta,\eta)\in M_0 \text{ if and only if } \eta =\lambda \bar{\zeta} \text{ for some $\lambda\ge 1$}
\end{equation}
and similarly 
\begin{equation}
        (\zeta,\eta)\in M_1 \text{ if and only if } \eta =1+\lambda (\bar{\zeta}-1) \text{ for some $\lambda\ge 1$}.
\end{equation}
$E^\zeta$ divides the plane in two regions, and we put the  orientation on $E^\zeta$ as the boundary of the region containing the real line.   
With this choice of orientation we define for $(\zeta,\eta)\notin M_0\cup M_1$
\begin{equation}\label{22}
    F(\zeta,\eta):= \frac{1}{2\pi i} \int_{E^\zeta} \frac{\tilde{h}(\zeta,w)}{w-\eta}\, dw . 
\end{equation}
We note that the integral converges in spite of the fact that $E^\zeta$ is not bounded; in fact for instance we have that on $E^\zeta_0$ 
$$ \tilde{h}(\zeta,\lambda\bar{\zeta})=\frac{1}{2\pi i}\int_{|\tau|=\sqrt{\lambda}|\zeta|} \frac{h(\tau)}{\tau-\zeta}\,d\tau =O(e^{-\sqrt{\lambda}}) $$
and the last equality follows from \eqref{21bis}.This is enough to ensure that the integral converges.
Clearly $F$ is holomorphic in $\eta$, it remains to prove that it is holomorphic in $\zeta$. To this end, we adopt the technique of Hans-Lewy to prove the holomorphic regularity of $F$ by the Morera theorem. For $(\zeta_0,\eta_0)\notin M_0\cup M_1$, keeping $\eta$ fixed, let $\gamma(t)=c+\varepsilon e^{it}$ be a small circle in the $\zeta$-plane in a neighborhood of $\zeta_0$ and  let $\Gamma$ be the disc with boundary $\gamma$. We consider
\begin{equation}\label{21}
	\int_\gamma F(\zeta,\eta_0)\,d\zeta =\int_\gamma d\zeta\int_{E^\zeta} \frac{\tilde{h}(\zeta,w)}{w-\eta_0}\,dw
\end{equation} 
and prove that this integral is $0$. We decompose $E^\zeta=E^\zeta_0\cup E^\zeta_1$ and treat each part separately. 
Let 
$$ V^{c,\varepsilon}_0:=\bigcup_{0\le t\le 2\pi} E^{\gamma(t)}_0 $$
be the set parametrized by $\varphi_0:[0,2\pi]\times [1,+\infty) \rightarrow V^{c,\varepsilon}_0$,  $ \varphi_0(t,s)=(c+\varepsilon e^{it},s(\bar{c}+\varepsilon e^{-it}))$ and we put on $V^{c,\varepsilon}_0$ the orientation induced by $\varphi_0$. Let $\iota(\Gamma)$ be the disc in $\triangle$ charted by $(\zeta,\bar{\zeta}), \zeta\in \Gamma$. We have that $V_0^{c,\varepsilon}\cup \iota(\Gamma)$ is a boundary in $M_0$ which means that there is $W_0\subset M_0$ such that $\p W_0= \iota(\Gamma)\cup V_0^{c,\varepsilon}$.
This set $W_0$ is foliated by complex leaves in the following way: for every $R$ let $\zeta\eta = R^2$ be the equation of one of the complex leaves forming $M_0$. We study the intersection of this leaf with $\iota(\Gamma)\cup V^{c,\varepsilon}_0$ we have $(\zeta,\eta)\in V^{c,\varepsilon}_0$ if and only if 
$$ \zeta=c+\varepsilon e^{it}, \text{  }\eta=\frac{R^2}{c+\varepsilon e^{it}}=\frac{R^2(\bar{c}+\varepsilon e^{-it})}{|c+\varepsilon e^{it}|^2} \text{ with } |c+\varepsilon e^{it}|^2\le R^2. $$
We see that if $R^2\ge |c|^2+\varepsilon^2$ we get a curve which bounds the disc $|\zeta -c|\le \varepsilon, \eta=\frac{R^2}{\zeta}$, the disc is contained in $M_0$.
If $ |c|^2 - \varepsilon^2 \le R^2\le |c|^2+ \varepsilon^2$ the intersection of the complex curve $\zeta\eta=R^2$ with $V^{c,\varepsilon}_0$ is an open curve whose endpoints are in $\triangle$. We close this curve with an arc of the circle $\zeta\eta=R^2, \eta=\bar{\zeta}$ which is inside $\iota(\Gamma)$.   
 By Stokes formula we have


\begin{equation}
	\int_\gamma d\zeta\int_{E^\zeta_0} \frac{\tilde{h}(\zeta,w)}{w-\eta_0}\,dw -\iint_{\iota(\Gamma)}  \frac{\tilde{h}(\zeta,w)}{w-\eta_0}\,d\zeta dw=\int_{W_0} \dib \left(\frac{\tilde{h}(\zeta,w)}{w-\eta_0}\right) \,d\zeta dw =0  
\end{equation} 
where the last equality follows because $\widetilde{H}$ is CR on $M_0$.
For the same reason we have a similar formula for $E^\zeta_1$ and since the term of integration on $\iota(\Gamma)$ appears on both terms, we have
\begin{equation}
	\int_\gamma\int_{E^\zeta}\frac{\tilde{h}(\zeta,w)}{w-\eta_0} \, d\zeta dw =\int_{V_0}\frac{\tilde{h}(\zeta,w)}{w-\eta_0}
\, d\zeta dw -\int_{V_1}\frac{\tilde{h}(\zeta,w)}{w-\eta_0}
\, d\zeta dw=0
\end{equation}
which proves that $F$ is holomorphic in $\zeta$ and $\eta$. We note that in \eqref{22} when $\zeta$ tends to the real axis, for $\zeta>1$ or $\zeta<0$, we have that the two lines $E^\zeta_0$ and $E^\zeta_1$ overlap hence $F$ tends to $0$. Since $F$ tends to $0$ on a generic subset we have that $F$ is identically $0$.
By the Plemelij-Schotty formula, we have that  $\tilde{h}$ is $0$ on $M_0\cup M_1$ and hence $h(\zeta)=\tilde{h}(\zeta,\bar{\zeta})=0\ \forall \zeta$ .  
\end{proof}
\section{End of the proof of Theorem \ref{t1}}\label{S5}
We need the following 
\begin{lemma}[Phragmen-Lindel{\"o}f on a sector]\label{l5}
    Let $U$ be the half plane $\Re( w)>0$. Let $h$ be a continuous function on the closure of $U$ and holomorphic on $U$. Assume there exists two constants $C>0$ and $\alpha<1$ such that
    \begin{equation*}
        |h(w)|\le Ce^{|w|^\alpha} \text{ for all } w\in U.
    \end{equation*}
    If $|h|\le 1$ on the imaginary axis then $|h|\le 1$ on $U$. 
\end{lemma}
for a proof we refer to \cite{L99}.
\begin{proof}[Proof of Theorem \ref{t1}]
By Proposition \ref{p4} we have that for all $(z_1,z_2)\in \mathbb{S}^3$ :
\begin{equation}
	\int_\R gf\left( \frac{z_1}{iy(z_2-1)+1},1+\frac{(z_2-1)}{iy(z_2-1)+1}\right)  dy =0
\end{equation}	
We can rewrite this integral using the following change of variable $\tau=\frac{1}{\zeta(z_2-1)+1}$ from which $\zeta=\left( \frac{1}{(z_2-1)\tau} -1\right) $ and $d\zeta =-\frac{1}{\tau^2 (z_2-1)} d\tau$. The imaginary line in the $\zeta$-variable is mapped to a circle $\mathcal{C}$ through $0$ in the $\tau$-variable and we have 
\begin{equation*}
	\int_\mathcal{C} gf(\tau z_1,1+(z_2-1)\tau)\left( \frac{1}{\tau^2(z_2-1)}\right)\, d\tau =0
\end{equation*}
If we apply the same reasoning to $(z_2-1)^kf(z_1,z_2)$ we have that
\begin{equation} \label{23}
\int_\mathcal{C} gf(\tau z_1,1+(z_2-1)\tau)(z_2-1)^{k-1}\tau^{k-2}\, d\tau =0	
\end{equation}
and since \eqref{23} holds for all $k$ we have that $gf$ extends holomorphically on the disc through $(z_1,z_2)$ and $(0,1)$ for all $(z_1,z_2)\in \mathbb{S}^3$. To prove that $f$ extends holomorphically we have to divide by $g$. Let $A_z:\D \rightarrow \B^2$ be the straight analytic disc through $(0,1)$ and $z=(z_1,z_2)$ such that $A_z(1)=(0,1)$ (we assume that $A_z(\tau)$ is linear in $\tau$). Let $gf(A_z(\tau))=F(\tau)$ for $\tau\in \p\D$ where $F\in \text{Hol}(\D)\cap C^0(\overline{\D})$. We consider the function on $\D$ defined by $\f(\tau):=\frac{F(\tau)}{g(A_z(\tau))}.$ We have that $\f$ is continuous on $\overline{\D}\setminus\{ 1\}$ and that $\f(\tau)=f(A_z(\tau))$ for $\tau \in \p\D\setminus\{ 1\}$. It remains to prove that $\f$ extends continuously at $1$. To this end we consider the change of variable $\tau=\varphi(w)=\frac{w-1}{w+1}$ which interchanges the unit disc $\D$ with the half-plane $\Re(w)>0$ and sends $\tau=1$ to $\infty$. We note that since $\frac{1}{|g(A_z(\varphi (w)))|}\le C_1e^{C_2|w|^\frac{1}{2}}$, it follows that
\begin{equation*}
    |\f(\varphi(w))|\le C_3 e^{C_2|w|^\frac{1}{2}}
\end{equation*}
where $C_1$, $C_2$ and $C_3$ are constants. Since $\f$ is bounded on $\p \D\setminus\{ 1\}$ by the supremum of $f$ we have by Lemma \ref{l5} that $\f(\varphi(\cdot))$ is uniformly bounded on the closed half-plane, hence $\f$ is uniformly bounded on $\D$. It follows that $\f$ has a boundary value on $\p\D$ which coincides with the continuous function $f(A_z(\tau))$ and we have the conclusion.
\end{proof} 
Thanks to Theorem \ref{t1} we are now in position to apply \cite{G12} and prove Corollary \ref{C1}. 
\begin{proof}[Proof of Corollary \ref{C1}]
After a rotation, assume that $p=(0,1)$. By Theorem \ref{t1} we have that $f$ extends holomorphically on the lines concurrent to $p$. The line joining $p$ and $c$ intersects the ball $\B$ so we can apply Corollary 1.3 of \cite{G12} and so $f$ is the boundary value of a holomorphic function $F\in \text{Hol}(\B^2)$. 
\end{proof}

 								\bibliographystyle{alpha}

\begin{thebibliography}{ST12}
\bibitem{A11} Agranovsky, M. Analog of a theorem of Forelli for boundary values of holomorphic functions on the unit ball of $\C^n$. {\em J. Anal. Math.} {\bf 113} (2011), 293–304.  
\bibitem{B13} Baracco, L. Separate holomorphic extension along lines and holomorphic extension from the sphere to the ball. {\em Amer. J. Math.} {\bf 135} (2013), no. 2, 493–497. \bibitem{B16} Baracco, L. Holomorphic extension from a convex hypersurface. {\em Asian J. Math.} {\bf 20} (2016), no. 2, 263–266.   
\bibitem{BF19}  Baracco, L.; Fassina, M. Orthogonal testing families and holomorphic extension from the sphere to the ball. {\em Math. Z.} {\bf 293} (2019), no. 3-4, 1277–1285.  	
\bibitem{BP18} Baracco, L.; Pinton, S. Testing families of complex lines for the unit ball. { \em J. Math. Anal. Appl.} {\bf 458} (2018), no. 2, 1449–1455.

\bibitem{G12} Globevnik, J. Meromorphic extensions from small families of circles and holomorphic extensions from spheres. {\em Trans. Amer. Math. Soc.} {\bf 364} (2012), no. 11
\bibitem{G12bis} Globevnik, J. Small families of complex lines for testing holomorphic extendibility. {\em Amer. J. Math.} {\bf 134} (2012), no. 6, 1473–1490. 

\bibitem{K} Komatsu, H. A local version of Bochner's tube theorem. {\em J. Fac. Sci. Univ. Tokyo Sect. IA Math.} {\bf 19} (1972), 201–214
\bibitem{L07} Lawrence, M. Hartogs' separate analyticity theorem for CR functions. {\em Internat. J. Math.} {\bf 18} (2007), no. 3, 219–229.
\bibitem{L18} Lawrence, M. The $L^p$ CR Hartogs separate analyticity theorem for convex domains. {\em Math. Z.} {\bf 288} (2018), no. 1-2, 401–414.
\bibitem{L99} Lang, S. Complex analysis. Fourth edition. {\em Graduate Texts in Mathematics}, {\bf 103} Springer-Verlag, New York, 1999.
\bibitem{T98} Tumanov, A. Analytic discs and the extendibility of CR functions. Integral geometry, Radon transforms and complex analysis (Venice, 1996), 123–141, {\em Lecture Notes in Math.}, {\bf 1684}, Springer, Berlin, (1998).
\bibitem{T97} Tumanov, A. Propagation of extendibility of CR functions on manifolds with edges. Multidimensional complex analysis and partial differential equations (São Carlos, 1995), 259–269, {\em Contemp. Math.}, {\bf 205}, Amer. Math. Soc., Providence, RI, (1997).

\bibitem{T07} Tumanov, A. Testing analyticity on circles.{\em Amer. J. Math.} {\bf 129} (2007), no. 3, 785–790.
 								
\bibitem{R} Rudin, W. ; Function theory in the unit ball of Cn. {\em Reprint of the 1980 edition. Classics in Mathematics.} Springer-Verlag, Berlin, (2008).






 

 							
 							\end{thebibliography}

 							\end{document}